\documentclass[a4paper,12pt]{amsart}

\usepackage{amssymb}
\usepackage{amsmath,amsthm}
\usepackage{enumerate,paralist}
\usepackage{amstext}
\usepackage{dsfont}
\usepackage[english]{babel}
\usepackage{color}
\usepackage{hyperref}
\usepackage{textgreek}
\usepackage{easyReview}
\usepackage{marginnote}

\usepackage[a4paper, left=2.5cm, right=2.5cm, top=3cm, bottom=3cm]{geometry}
\usepackage{mathtools}
\usepackage{latexsym}
\usepackage{mathrsfs}
\usepackage{MnSymbol}
\usepackage{bbm}

\theoremstyle{plain}
\newtheorem{theorem}{Theorem}

\newtheorem{lemma}[theorem]{Lemma}
\newtheorem{proposition}[theorem]{Proposition}
\theoremstyle{definition}
\newtheorem{definition}[theorem]{Definition}

\newtheorem{remark}[theorem]{Remark}

\numberwithin{theorem}{section}

\DeclareMathOperator*{\esssup}{ess\,sup}
\DeclareMathOperator*{\essinf}{ess\,inf}
\DeclareMathOperator*{\argmin}{arg\,min}
\newcommand{\dive}{{\operatorname{div}}}

\newcommand{\qtext}[1]{\quad\mbox{#1}\quad}
\newcommand{\qqtext}[1]{\qquad\mbox{#1}\qquad}
\newcommand{\pf}{{}_{\#}}

\newcommand{\R}{\mathbb{R}}

\newcommand{\mrest}{
  \,\raisebox{-.127ex}{\reflectbox{\rotatebox[origin=br]{-90}{$\lnot$}}}\,%
}
\newcommand{\1}{\mathds 1}

\newcommand{\eps}{\varepsilon}

\newcommand{\fF}{\mathcal{F}}
\newcommand{\PP}{\mathbb{P}}
\newcommand{\dD}{\mathcal{D}}
\newcommand{\eE}{\mathcal{E}}
\newcommand{\pP}{\mathcal{P}}

\newcommand{\qQ}{\mathcal{Q}}
\newcommand{\lL}{\mathcal{L}}
\newcommand{\uL}{\underline{L}}
\newcommand{\mM}{\mathcal{M}}
\newcommand{\hH}{\mathcal{H}}

\newcommand{\wW}{\mathcal{W}}
\renewcommand{\O}{\Omega}
\renewcommand{\o}{\omega}
\newcommand{\pO}{{\partial\Omega}}
\newcommand{\Oi}{{\Omega^{\mathrm o}}}
\newcommand{\g}{\gamma}
\newcommand{\rd}{\mathrm d}
\newcommand{\Ent}{\mathrm{Ent}}
\newcommand{\Lip}{\mathrm{Lip}}
\newcommand{\grad}{\mathrm{grad}}
\newcommand{\tr}{\operatorname{tr}}
\newcommand{\D}{\mathcal D}

\def\la{\lambda}

\def\Del{\Delta}

\def\si{\sigma}

\def\eps{\varepsilon}

\def\Om{\Omega}
\def\dOm{\partial \Omega}

\def\pd{\partial}

\def\lp{\left(}
\def\rp{\right)}

\def\hoch{^}

\newcommand{\Iint}{{\mathrm o}}

\begin{document}

\title[]{A gradient flow that is none: heat flow with Wentzell boundary conditions}

\author[M. Bormann]{Marie Bormann}
\address{Universit\"at Leipzig, Fakult\"at f\"ur Mathematik und Informatik, Augustusplatz 10, 04109 Leipzig, Germany and Max Planck Institute for Mathematics in the Sciences, 04103 Leipzig, Germany}
\email{bormann@math.uni-leipzig.de}

\author[L. Monsaingeon]{L\'{e}onard Monsaingeon}
\address{Grupo de F\'isica Matem\'atica, Departamento de Matem\'atica, Instituto Superior T\'ecnico, Av. Rovisco Pais 1049-001 Lisboa, Portugal
and
Institut \'Elie Cartan de Lorraine, Universit\'e de Lorraine, Site de Nancy B.P. 70239, F-54506 Vandoeuvre-l\`es-Nancy Cedex, France
}
\email{leonard.monsaingeon@tecnico.ulisboa.pt}

\author[D.R.M. Renger]{D.R. Michiel Renger}
\address{Technische Universit\"at M\"unchen, School of CIT, Boltzmannstra{\ss}e 3, 85748 Garching bei München}
\email{d.r.m.renger@tum.de}

\author[M. von Renesse]{Max von Renesse}
\address{Universit\"at Leipzig, Fakult\"at f\"ur Mathematik und Informatik, Augustusplatz 10, 04109 Leipzig, Germany}
\email{renesse@math.uni-leipzig.de}

\date{\today}

\begin{abstract} We establish a  representation of the heat flow with Wentzell boundary conditions on smooth domains as gradient descent dynamics for the entropy in a suitably extended Otto manifold of  probability measures with  additional boundary parts. Yet it is shown that for weak boundary diffusion, the associated Fokker–Planck dynamics cannot be recovered from any entropy-driven metric JKO-Wasserstein scheme, at least if the underlying point metric satisfies certain natural regularity assumptions.  This discrepancy is illustrated in competing large-deviation heuristics in the Sanov and Schilder regimes.
    \end{abstract}
\maketitle
\section{Introduction}

In a short but  influential early career paper~\cite{MR1477263}  Karl-Theodor Sturm showed that a reversible diffusion process with known invariant measure is in general not uniquely characterised by its intrinsic metric. Given the  success of optimal transport theory over the last two decades, one might modify Sturm's question and wonder which reversible diffusion processes with fixed invariant measure can be constructed via a  Wasserstein gradient flow for the Boltzmann-entropy and \textit{some}  metric $d$. Identifying a gradient flow structure for the associated Fokker-Planck equation is desirable for the application of existing theory for gradient flows and for additional physical and analytical insight. In this note we point out that such a metric might not exist, in general.

\smallskip

To this aim, we discuss the heat flow on a bounded domain with so-called Wentzell boundary conditions~\cite{MR121855}, which appears also as the Fokker-Planck equation for Brownian motion with sticky-reflecting boundary diffusion.
One characteristic feature of this process is an invariant measure with two mutually singular components on the interior and on the boundary of the domain.
Even though the process is reversible, we will show in the weak boundary diffusion regime that the associated heat flow cannot admit a JKO-style Wasserstein gradient flow representation \cite{JKO98,MR2401600} for any metric on the base space satisfying certain natural regularity conditions (as detailed in section \ref{section:metricsection} below).
Interestingly, the heat flow can nevertheless be represented (at least formally) as a \textit{gradient system} in the sense of the framework \cite{MR3269725} extending  the well-known Otto calculus \cite{MR1842429} to a wider class of dissipative evolution equations.
In particular, this provides an example of a class of reversible diffusion processes in which the differential Otto calculus resolves all details of the dynamics whereas the purely metric Wasserstein JKO-method encounters limitations.

\smallskip 

The differences in the two approaches are also reflected in two distinctive heuristics which are commonly used for the identification of the sought-after gradient-flow structure and which are both based on large-deviation arguments. The metric picture is often justified from its relation to the  Schr\"{o}dinger problem in conjunction with a Schilder-type large-deviation principle for slowed-down diffusion processes, leading to a Varadhan-formula for the intrinsic metric \cite{MR2873864}, whereas the gradient system representation is deduced from a dynamical Sanov-type large-deviation principle for the hydrodynamic many-particle limit cf.~\cite{adams2013large,MR3269725}. In the regime of strong boundary diffusion, both heuristics lead to the same metric gradient flow representation, but the coincidence breaks down if the boundary diffusion falls below a certain threshold, leading to  a parameterized family of distinct diffusion processes with common intrinsic metric  and common invariant measure. A~key aspect of this phase transition is the loss of joint lower semicontinuity  of the Lagrangian function for the associated Hamilton-Jacobi equation and their viscous approximations, which lie at the heart of the theory, pointing to some of the fundamental challenges of the metric approach in very irregular settings.

\section{Our Example: Brownian motion with sticky-reflecting boundary diffusion (SRBD-Brownian Motion)}

The basic object of this study is the family of Brownian motions on bounded domains with sticky-reflecting boundary diffusion. Let $\Om$ be a compact connected subset of $\mathbb{R}^d, d\ge 2$ with nonempty interior $\Oi$ and smooth connected boundary $\dOm$. We consider the semigroup  on $C(\Om)$ induced by the  Feller generator $(\dD(\qQ),\qQ)$ given by
\begin{align}\label{eq:defgen}
\dD(\qQ) &= \{f\in C(\Om)\ |\ \qQ f\in C(\Om)\} \notag \\
\qQ f &= \Del f\1_{\Oi} + \lp a\Del^\tau f - \pd_N f\rp\1_{\dOm},
\end{align}
where $\pd_N f$ is the outer normal derivative, $\Del^\tau$ is the Laplace-Beltrami operator on $\dOm$ and $a > 0$. This generator defines a Markov process on $\Om$ with continuous trajectories, which performs a Brownian motion in the interior $\Oi$, a  tangential Brownian diffusion along the boundary $\dOm$, where the strength of the tangential diffusion is given by the parameter $a\geq 0$, and is subject to 'sticky reflection' from the boundary back into the interior, resulting ultimately in a positive sojourn time on the boundary\footnote{We wish to avoid a factor $1/2$ in the infinitesimal generator. Thus, to be precise, the underlying Brownian motion both in the interior and on the boundary is sped up by a factor 2.}.

\smallskip 
Denoting the Lebesgue measure restricted to $\Om$ by $\lambda$ and the Hausdorff measure on $\dOm$ by $\sigma$, choosing $c=1/(\la(\O)+\sigma(\pO))$ makes
\begin{equation*}
\mu := c (\la + \si)
\end{equation*}
a probability measure, and it is easy to check that $-\qQ$ is a positive $\mu$-symmetric operator.
In particular, the process is reversible with respect to the invariant measure $\mu$.
The corresponding Dirichlet form is $(\dD(\eE),\eE)$
\begin{equation}\label{eq:DF}
\eE(f):= \int_\Oi |\nabla f|^2 c\,\rd\la+ a \int_{\dOm} |\nabla^\tau f|^2 c\,\rd\si
\end{equation}
 with $\nabla^\tau  f$ the tangential gradient along $\pO$, and the domain $\dD(\eE)$ is the closure of $C^1(\Om)$ under the norm $\|\cdot\|_{\dD}^2:=\eE(\cdot)+ \|\cdot\|^2_{L^2(\mu)}$, see~\cite{MR3613700}.

The abstract Fokker-Planck equation $\partial_t\rho=\qQ^*\rho$ is, by definition, $\frac{d}{dt}\int_\O \varphi \rho=\int_\O (\qQ\varphi) \rho=\int_\O\varphi (\qQ^*\rho)$ for all $\varphi\in \dD(\qQ)$.
Here and throughout it will be useful do decompose arbitrary measures $\rho$ on $\Om$ in interior/boundary parts
\begin{equation*}
\rho=\omega + \gamma \in\pP(\Om) 
\qqtext{with}
\begin{cases}	\omega\coloneqq\rho \mrest_{\Oi}, \\ \gamma\coloneqq\rho\mrest_{\dOm},\end{cases}
\end{equation*}
and we often identify the measures $\o,\g$ with their density w.r.t. the measures $\lambda,\sigma$ in the interior and on the boundary. Using this representation and integrating by parts, we compute the Fokker-Planck equation explicitly:
\begin{align*}
    \int_\Oi\varphi\partial_t\o + \int_\pO\varphi\partial_t\g
    &=\int_\O\varphi\partial_t\rho
    =\frac{d}{dt}\int_\O \varphi \rho
    =\int_\O (\qQ\varphi)\rho
    \\
    &=\int_\Oi\Delta\varphi\, \o + \int_\pO (a\Delta^\tau\varphi - \partial_N\varphi)\,\g
    \\
    &=\int_\Oi\varphi \Delta\o + \int_\pO\partial_N\varphi\,\o - \varphi\partial_N\o + \int_\pO \varphi a\Delta^\tau\g-\partial_N\varphi\,\g
    \\
    &=\int_\Oi\varphi \Delta\o + \int_\pO\varphi(a\Delta^\tau\g-\partial_N\o)  + \int_\pO \partial_N\varphi(\o-\g)
\end{align*}
for all $\varphi$, which is of course the weak formulation of the heat flow with boundary conditions:
\begin{equation}\label{eq:FPE}
\begin{cases} \partial_t\omega = \Delta \omega &\text{ in } \Oi, \\ 
	\omega = \gamma &\text{ on } \dOm, \\ 
	\partial_t\gamma = a\Delta^\tau \gamma - \pd_N \omega &\text{ on } \dOm.
\end{cases}
\end{equation}
The sort of boundary behaviour considered here first came up in a work of Wentzell~\cite{MR121855}, and is accordingly also referred to as \textit{Wentzell boundary condition}. In the sequel we shall use the name \textit{Wentzell heat flow} for this evolution.

First rigorous constructions of such processes on special domains date back to~\cite{MR126883,MR287612,MR929208}. An efficient construction via Dirichlet forms was given recently in \cite{MR3613700}.
Renewed interest has arisen lately from applications in interacting particle systems with mean-field or zero-range pair interaction~\cite{MR4096131,MR4575023,MR4704529}.  Poincar\'{e} and logarithmic Sobolev inequalities for SRBD-Brownian motion are derived in in~\cite{konarovskyi2021spectralgapestimatesbrownian, brw, bormann2025cheegertypeinequalitydriftlaplacian}, in spite of the fact that there is no finite lower (generalized) Ricci curvature bound.   The JKO-Wasserstein gradient flow structure in the regime of regular boundary diffusion speed ($a=1$) has been investigated in \cite{MR4901547} and a Schilder-type large-deviation principle has been studied recently in~\cite{casteras2025largedeviationsstickyreflectingbrownian}.

\section{Wentzell heat flow  as Gradient system}\label{sec:gradsyst}
In this section we aim at identifying the Wentzell heat flow~\eqref{eq:FPE} as a differential gradient flow.
Here we do not claim full mathematical rigor and completely disregard smoothness issues.
We first discuss which infinitesimal perturbations $\dot\rho$ should give admissible kinematics, and how to measure kinetic energy, taking into account the boundary coefficient $a>0$.
This allows to construct a pseudo-Riemannian structure \textit{\`a la Otto} \cite{MR1842429}.
This point of view provides an infinitesimal differential structure for each value of $a>0$, and identifies in particular the Wentzell heat flow as the gradient-flow
\begin{equation}
\label{eq:FPE_iff_grad_Ent}
\partial_t\rho=\qQ^*\rho
\quad\iff\quad
\dot\rho=-\grad\,\Ent_\mu(\rho)
\end{equation}
of the relative entropy
$$
\Ent_\mu(\rho)=\int_\O\frac{\rd\rho}{\rd\mu}\log\left(\frac{\rd\rho}{\rd\mu}\right)\rd\mu
=
\int_\Oi \o\log\o +\int_\pO\g\log\g.
$$
(In this paper we write $\partial_t\rho$ for partial derivatives with respect to time in PDEs, while the very same object is rather denoted $\dot\rho$ or $\frac{d\rho}{dt}$ when viewed as a tangent vector in this Otto formalism.)
Interestingly enough, boundary conditions play a particular role and the trace coupling $\o|_\pO=\g$ will be necessary to actually define the object $\grad\,\Ent_\mu(\rho)$ by chain rule and duality.
Once this is done we will discuss the connection with large-deviation principles for microscopic particle systems.
As suggested in \cite{MR3269725}, and at least for smooth settings, the reversibility of the process is a sufficient condition for the corresponding Fokker-Planck equation to possess a $\Psi/\Psi^*$ gradient flow structure.
We will show that, despite the singular nature of our problem, this remains essentially valid, with however a few twists due to the peculiar boundary conditions.

\subsection{Admissible kinematics}
Let us point out that the infinitesimal variations $s=\dot\rho\in \mM(\O)$ are always dual to functions $\phi$, with canonical pairing $\langle \phi,s\rangle=\int \phi\, \rd s$.
Moreover, given that our stochastic processes and densities behave differently in the interior and along the boundary, we think of $\dot\rho =s\in \mM(\O)$ as pairs $s=(s_\Iint,s_\partial)$ acting in $\Oi,\pO$.

As always we enforce local conservation of mass, which means in particular that a particle hitting the boundary from the interior with velocity $V_\Iint$ does not exit the domain, but can stick to the boundary and start evolving thereupon with a new tangential velocity $V_\partial$.
From an Eulerian viewpoint, conservation of mass across the boundary simply means that the outflux of interior densities corresponds to a source term for the boundary density in the equations of motion.
In formulas, this means that we think of admissible kinematics $\dot\rho=s$ as
\begin{equation}
\label{eq:kinematics}
s=(s_\Iint,s_\partial)
\qqtext{s.t.}
\begin{cases}
s_\Iint=-\dive(\o V_\Iint) & \text{in }\Oi\\
s_\partial=-\dive^\tau(\g V_\partial)+\o V_\Iint\cdot N & \text{in }\pO
\end{cases}
\end{equation}
for some interior/boundary pair of velocity fields $V=(V_\Iint,V_\partial)$ (with $V_\partial$ tangent to $\pO$).
In a more compact form this simply means
$$
\langle s,\xi\rangle = \int_\Oi \o V_\Iint\cdot \nabla \xi + \int_\pO \g V_\partial\cdot \nabla^\tau \xi
= \int_\O \rho V\cdot \nabla \xi.
$$
This is exactly the standard weak formulation appearing in the usual continuity equation $\partial_t\rho + \dive(\rho V)=0$ for the total density $\rho=\o+\g$ with no-flux boundary conditions.
Indeed in this simple picture there is no exchange of mass between $\O$ and the outer world $\O^c$: particles are free to move between $\Oi$ and $\pO$, but any mass initially contained in $\O=\Oi\cup\pO$ remains so.
The fact that we are considering here $\rho=\o+\g$ with singular parts $\g$ on the boundary simply suggests considering $V$ as two separate, independent parts $V=(V_\Iint,V_\partial)$.

We now \emph{choose} to measure kinetic energy as
\begin{equation}
\label{eq:def_kinetic_E}
\|V\|^2_\rho 
\coloneqq
\int_\Oi |V_\Iint|^2\o + \frac 1a\int_\pO|V_\partial|^2\g.
\end{equation}
The difference between interior and boundary contributions is reflected in the local  $\frac 1a$ scaling of the Euclidean metrics on $\pO$.
Similarly, for smooth enough functions $\phi$ we set
\begin{equation}
\label{eq:def_|phi|_rho}
\|\phi\|^2_\rho
\coloneqq
\int_\Oi |\nabla\phi|^2\o +a\int_\pO|\nabla^\tau\phi|^2\g.
\end{equation}
Note that the Dirichlet form \eqref{eq:DF} accordingly reads
$$
\eE(\xi)=\|\xi\|^2_\mu,
$$
where we recall that $\mu\in\pP(\O)$ is the stationary, symmetric measure for SRBD Brownian Motion.
All of this is of course very consistent with the initial interpretation of $a>0$ as a boundary diffusion coefficient in \eqref{eq:defgen}\eqref{eq:FPE}: the larger $a$ the stronger the diffusion, and larger $a$'s accordingly correspond to cheaper displacement along the boundary according to \eqref{eq:def_kinetic_E}.
Since many $V$'s possibly represent the same $s$ one natural selection principle consists in minimizing the kinetic energy:
\begin{proposition}
Let $s=(s_\Iint,s_\partial)$ be represented by at least one admissible pair $V=(V_\Iint,V_\partial)$ in the sense of \eqref{eq:kinematics}, with $\|V\|^2_\rho <\infty$.
Then
$$
\|s\|_\rho^2
\coloneqq
\min\limits_{V=(V_\Iint,V_\partial)}\Bigg\{ \|V\|^2_\rho
\quad \text{s.t.\ the representation }
\eqref{eq:kinematics}\text{ holds}
\Bigg\}
$$
is given by
\begin{equation}
\label{eq:s_phi_norms}
\|s\|_\rho^2
=\int_\Oi|\nabla\phi|^2\o +a\int_\pO|\nabla^\tau\phi|^2\g
=\|\phi\|^2_\rho
\end{equation}
where $\phi$ is the unique solution (up to additive constants) of
\begin{equation}
\label{eq:elliptic_id}
\begin{cases}
s_\Iint=-\dive(\o\nabla\phi) & \text{in }\Oi\\
s_\partial=-a\dive^\tau(\g\nabla^\tau\phi)+\o\partial_N\phi & \text{in }\pO
\end{cases}.
\end{equation}
\end{proposition}
This is a rather standard $H^{-1}_\rho/\dot H^1_\rho$ elliptic identification between admissible perturbations $s=\dot\rho$ and scalar potentials $\phi=\phi_s$, except that the corresponding elliptic system couples $\Oi$ to $\pO$ via the implicit boundary conditions for $\phi$: the $\phi$ appearing in the tangential elliptic PDE along $\pO$ in \eqref{eq:elliptic_id} must be the boundary trace $\phi|_{\pO}$ of $\phi$ appearing in the interior PDE inside $\Oi$.

Note carefully that any admissible $s$ will have finite norm $\|s\|_\rho<+\infty$, regardless of whether any trace condition on $\o|_\pO,\g$ is satisfied or not.
We draw the reader's attention to our clear abuse of notation here: we write indistinctly $\|\cdot\|_\rho$ for the $L^2(\rho)$ norm of  velocity fields $V$, the $\dot H^1(\rho)$ norm of functions $\phi$, and for $H^{-1}(\rho)$ norms of tangent perturbations $s=\dot\rho$.
We hope that the context should make it clear which is which.
\begin{proof}
Given that the constraint \eqref{eq:kinematics} is affine and that the kinetic energy is strictly convex clearly the problem has a unique solution $V^*=(V^*_\Iint,V^*_\partial)$.

Perturbing first $V^\eps_\Iint=V^*_\Iint+\eps \frac{\eta_\Iint}{\o}$ for any compactly supported field with $\dive \,\eta_\Iint=0$, one classically obtains that $V^*_\Iint=\nabla\phi_\Iint$ for some $\phi_\Iint$.
Similarly, perturbing $V^\eps_\partial=V^*_\partial+\eps \frac{\eta_\partial}{\g}$ with $\dive^\tau\eta_\partial=0$ shows that $V^*_\partial=a\nabla^\tau \phi_\partial$ for some $\phi_\partial$.

The key point is to prove that the two potentials can be chosen to match across the boundary, i.e.\ $\left.\phi_\Iint\right|_{\pO}=\phi_\partial$.
To this end take any $\tilde f$ such that $\int_\pO\tilde f=0$ and define $\tilde\phi_\Iint,\tilde\phi_\partial$ by
$$
\begin{cases}
-\dive(\o\nabla\tilde\phi_\Iint)=0 &\text{ in }\O^\circ\\
\o\partial_N\tilde\phi_\Iint = \tilde f &\text{ on } \dOm
\end{cases}
\qqtext{and}
-\dive^\tau(\g \nabla^\tau\tilde\phi_\partial)=-\tilde f \text{ on } \dOm.
$$
By construction the perturbation
$$
V_\eps
=(V^*_\Iint +\eps \nabla\tilde\phi_\Iint,V^*_\partial+\eps \nabla^\tau \tilde\phi_\partial)
=(\nabla\phi_\Iint +\eps \nabla\tilde\phi_\Iint,a \nabla^\tau\phi_\partial+\eps \nabla^\tau \tilde\phi_\partial)
$$
still satisfies~\eqref{eq:kinematics} for any $\eps>0$, hence by optimality we must have
\begin{multline*}
0
= \left.\frac{d}{d\eps}\right|_{\eps=0}\frac 12 \|V^\eps\|^2_\rho
=
\int_\Oi \o\nabla\phi_\Iint\cdot \nabla\tilde\phi_\Iint+\frac 1a\int_\pO \g a\nabla^\tau\phi_\partial\cdot \nabla^\tau\tilde\phi_\partial
\\
=\Bigg(-\int_\Oi \phi_\Iint \underbrace{\dive(\o\nabla\tilde\phi_\Iint)}_{=0}+\int_\pO \phi_\Iint \underbrace{\o\nabla \tilde\phi_\Iint \cdot N}_{=\tilde f}\Bigg)
- \int_\pO \phi_\partial\underbrace{\dive^\tau(\g\nabla \tilde\phi_\partial)}_{=\tilde f}
\\
=\int_\pO (\phi_\Iint-\phi_\partial)\tilde f.
\end{multline*}
Since $\tilde f$ was arbitrary with $\int_\pO\tilde f=0$ we see that $\left.\phi_\Iint\right|_{\pO}=\phi_\partial+C$ for some constant $C$.
Moreover, $\phi_\Iint,\phi_\partial$ are only defined up to additive constants so we can normalize $C=0$ without loss of generality and the proof is complete.
\end{proof}

\subsection{Otto calculus representation of the Wentzell heat flow}
\label{ssec:Otto}
Recalling that $\mu=c(\lambda+\sigma)$ is the stationary, reversible measure, we define the ``manifold''
\begin{multline}
    \pP^*=\pP^*_\mu(\O)
    \coloneqq\Bigg\{\rho=\o+\g\ll\mu,\quad \rho\in \pP(\O)\\
    \quad \text{with }\o>0\text{ a.e. in }\Oi \text{ and }\g>0 \text{ a.e. in }\pO\Bigg\}
\end{multline}
and admissible ``tangent space''
$$
T_\rho\pP^*=\Big\{\dot\rho=s=(s_\Iint,s_\partial)\qquad\text{with elliptic identification }s\leftrightarrow\phi_s \text{ given by }\eqref{eq:elliptic_id}\Big\}.
$$
Here and as always we abused notations and identified the measures $\o=\rho\mrest_\Oi,\g=\rho\mrest_\pO$ and their densities $\o(x),\g(x)$ w.r.t. the measures $\lambda,\sigma$ on $\Oi,\pO$.
As suggested by \eqref{eq:s_phi_norms}, \eqref{eq:elliptic_id} we define the scalar product on $T_\rho\pP^*$ 
\begin{equation}
\label{eq:def_scal_prod_s}
    \langle s^1,s^2\rangle_\rho =
\int_\Oi \o\nabla\psi^1\cdot\nabla\psi^2 +a\int_\pO\g \nabla^\tau\psi^1\cdot\nabla^\tau\psi^2
=\int_\Oi \o V^1_\Iint\cdot V^2_\Iint +\frac 1a\int_\pO  \g V^1_\partial\cdot V^2_\partial
\end{equation}
and norm
\begin{equation}
\label{eq:def_norm_s}
\|s\|^2_\rho = \int_\Oi \o |\nabla\psi|^2+a\int_\pO\g |\nabla^\tau\psi|^2
=\int_\Oi \o|V_\Iint|^2+\frac 1a\int_\pO\g |V_\partial|^2,    
\end{equation}
where the $V^i$'s are the optimal representatives of $s^i$ through \eqref{eq:elliptic_id}.
\\
We now aim at computing gradients with respect to this Riemannian structure.
To this end take two functions: $O,G:\R^+\to\R$ and consider the internal energy
$$
\fF(\rho)\coloneqq \int_\Oi O(\o)+\int_\pO G(\g).
$$
Fix $\rho\in \pP^*$ and an arbitrary admissible tangent vector
$$
\dot\rho =
\begin{pmatrix}
\dot\o\\ \dot\g
\end{pmatrix}
=
\begin{pmatrix}
-\dive(\o\nabla\phi)
\\
-a\dive^\tau(\g\nabla^\tau\phi)+\o\partial_N\phi
\end{pmatrix}.
$$
Given any curve $\rho(t)\in \pP^*$ passing through $\rho=\rho(0)$ with speed $\dot\rho\in T_\rho\pP^*$, standard chain rule and integration by parts give \begin{multline}
\label{eq:chain_rule_F}
\left.\frac d{dt}\fF(\rho(t))\right|_{t=0}=
\int_\Oi O'(\o)\partial_t\o + \int_\pO G'(\g)\partial_t\g
\\
=-\int_\Oi O'(\o)\dive(\o\nabla\psi) -\int_\pO G'(\g)\,a\dive^\tau(\g\nabla^\tau\psi)
+\int_\pO G'(\g)\o\partial_N\psi
\\
=
\int_\Oi \o\nabla O'(\o)\cdot\nabla\psi + a \int_\pO \g \nabla^\tau G'(\g)\cdot \nabla^\tau\psi
+ \int_\pO \Big(G'(\g)-O'(\o)\Big)\o\partial_N\psi
\\
=\left\langle D_\rho\fF,\psi\right\rangle_{\rho}
+\int_\pO \Big(G'(\g)-O'(\o)\Big)\o\partial_N\psi,
\end{multline}
where the linear first variation is simply $D_\rho\fF=\1_{\Oi} O'(\o)+\1_{\pO}G'(\g)$.
Given the topology chosen on the tangent space, and since $\o>0$ in $\pP^*$, the right-hand side cannot be a $\|\cdot\|_\rho$-continuous linear form in $\psi$ unless the compatibility condition
$$
O'(\o)\big\vert_{\pO}=G'(\g)
$$
holds.
In other words, the functional $\fF$ is only differentiable (w.r.t. the above pseudo-Riemannian structure) at points $\rho\in \pP^*$ satisfying $O'(\o)\big\vert_{\pO}=G'(\g)$.
At those points, one simply reads off of \eqref{eq:chain_rule_F} that the Riemannian gradient $\grad\,\fF$ defined by the chain rule
$$
\left\langle\grad\,\fF(\rho),\dot\rho\right\rangle_\rho
=\left.\frac{d}{dt}\right|_{t=0}\fF(\rho(t))
\qtext{for arbitrary}\dot\rho\in T_\rho\pP^*
$$
is identified in $H^{-1}_\rho$ with the scalar potential $\phi=D_\rho\fF$.
In other words,
$$
\grad\,\fF(\rho) =
\begin{pmatrix}
    -\dive\left(\o\nabla O'(\o)\right)\\
    -a\dive^\tau \left(\g\nabla^\tau G'(\g)\right)
\end{pmatrix}
\qtext{whenever}O'(\o)|_\pO=G'(\g)
$$
and is undefined otherwise (or, rather, $\|\grad\,\fF\|_\rho=\infty$).

For the specific case of the relative entropy $\Ent_\mu$ we simply have $O(z)=G(z)=z\log z-z+1=H(z)$, hence upon requiring
$$
H'(\o)\big|_\pO=H'(\g)
\iff 
\log\o|_\pO=\log\g
\iff 
\o|_\pO=\g,
$$
we get that
$$
\grad\,\Ent_\mu(\rho) =
\begin{pmatrix}
    -\dive\left(\o\nabla\log(\o)\right)\\
    -a\dive^\tau \left(\g\nabla^\tau \log(\g)\right)+\o\partial_N \log(\o)
\end{pmatrix}
=
\begin{pmatrix}
    -\Delta\o\\
    -a\Delta^\tau \g+\partial_N\o
\end{pmatrix}.
$$
In view of \eqref{eq:FPE} this clearly means that we have the gradient flow structure \eqref{eq:FPE_iff_grad_Ent}.
We stress again that this holds for \emph{all} $a>0$, and  that the very existence of the object $\grad\,\Ent_\mu(\rho)$ is what enforces the trace condition $\o|_\pO=\g.$ 
Furthermore, we stress that the driving functional considered here is the ``usual'' relative entropy $\Ent_\mu$ regardless of the value of $a>0$, while $a$ does enter in the norm on the tangent space.
This is in contrast with other works on gradient flow structures for Fokker-Planck equations with boundary condition, such as \cite{quattrocchi2024variationalstructuresfokkerplanckequation}, where due to the boundary condition a modified entropy is considered.

\subsection{\texorpdfstring{$\Psi/\Psi^*$}{}-Formalism and Dynamical Sanov Large-Deviation Principle for SRBD-Brownian motion}
As mentioned above we now connect the gradient flow structure \textit{\`{a} la Otto} with large-deviations and the $\Psi/\Psi^*$ formalism from \cite{adams2013large,MR3269725}.
Starting from the path level, let $X_1(t),X_2(t),\ldots$ be a sequence of independent SRBD processes, each with generator $(\dD(\qQ),\qQ)$ given by \eqref{eq:defgen}.
By Varadarajan's Theorem~\cite[Th.~11.4.1]{Dudley2002}, the empirical process
\begin{equation*}
\rho^n:t\longmapsto \frac 1n\sum_{k=1}^n \delta_{X_k(t)}, 
\qquad t\in[0,T],
\end{equation*} 
converges as $n\to\infty$ to a deterministic solution $\rho:[0,T]\to\pP(\Om)$ of the associated Fokker-Planck equation, i.e. the Wentzell heat flow $\dot\rho=\qQ^\ast\rho$.
We are interested in a large-deviation principle of the form
\begin{equation*}
\PP((\rho^n_t)_{t=0}^T\approx \rho) \stackrel{n\to\infty}{\asymp} \exp(-n(I_0(\rho_0)+I_T(\rho))),
\qquad
I_T(\rho)=\int_0^T \lL(\rho_t,\dot{\rho}_t) \rd t.
\end{equation*}
Let us stress once and for all that we do not claim mathematical rigour at this stage, and the discussion about LDP below is intended to remain purely formal, see also Remark~\ref{rmk:no_LDP?} below.

As initiated by Dawson and G\"artner \cite{dawsont1987large} and following Feng and Kurtz~\cite{MR2260560}, the Lagrangian $\lL(\rho,s)$ can usually be obtained as the Legendre convex conjugate (in the second variable) of some Hamiltonian $\hH(\rho,\xi)$.
In our setup with independent copies, this Hamiltonian takes the form~\cite{MR3269725}:
\begin{equation}\label{eq:hdef}
  \hH(\rho,\xi):=\int_\Om e^{-\xi} Q e^\xi \,\rd\rho.
\end{equation}
Given that we only consider densities decomposing into interior/boundary parts $\rho=\o+\g$, and in view of the particular expression \eqref{eq:defgen} of the generator, we explicitly get 
\begin{align}
\hH(\rho,\xi) &
= \int _{\O^\circ} e^{-\xi}\Delta e^\xi \o + \int_\pO e^{-\xi}(a\Delta^\tau e^\xi -\partial_N e^\xi)\g
\notag
\\
&= 
\int_{\O^\circ} (\Delta\xi +|\nabla\xi|^2)\o 
+ a\int_\pO(\Delta^\tau \xi +|\nabla^\tau\xi|^2)\g - \int_\pO \partial_N\xi \g.
\label{eq:H_expl_SBM}
\end{align}
Let us recall that our infinitesimal variations $s=\dot\rho\in \mM(\O)$ are dual to functions $\xi$, with canonical pairing $\langle \xi,s\rangle=\int \xi \rd s=\int \xi s$.
We can therefore compute the Legendre transform as
\begin{multline*} 
\lL(\rho,s)=\hH^\ast(\rho,s)
=\sup\limits_\xi \, \langle \xi,s\rangle - \hH(\rho,\xi)
\\
=\sup\limits_\xi
\int_\O \xi s -\left(\int_{\O^\circ} (\Delta\xi +|\nabla\xi|^2)\o 
+ a\int_\pO(\Delta^\tau \xi +|\nabla^\tau\xi|^2)\g - \int_\pO \partial_N\xi \g\right).
\end{multline*}
Integrating by parts $\int_{\Oi}  \o\Delta\xi  =
\int_{\Oi}  \xi \Delta\o + \int_\pO \o\partial_N\xi-\xi\partial_N\o$
 and $
\int_\pO \g\Delta^\tau \xi  =\int_\pO \xi \Delta^\tau \g
$, and rearranging interior and boundary contributions, we get altogether
\begin{multline}
\label{eq:L_intermediate}
\lL(\rho,s)
=\sup\limits_\xi
\int_{\O^\circ} \xi \left(s-\Delta\o \right) -|\nabla\xi|^2\o
+
\int_{\pO} \xi\left(s-a\Delta^\tau\g  +\partial_N\o\right) -a|\nabla^\tau \xi|^2\g
+ \int_\pO \partial_N\xi(\g-\o)
\\
=\sup\limits_\xi\left[
\int_\O \xi(s-\mathcal Q^*\rho)
-\|\xi\|^2_\rho\right]
+ \int_\pO \partial_N\xi(\g-\o).
\end{multline}
Here $\|\xi\|_\rho$ was previously defined in \eqref{eq:def_|phi|_rho} and $\qQ^*\rho =\1_{\Oi}\Delta\o+\1_{\pO}(a\Delta^\tau\g -\partial_N\o)$ is the dual Fokker-Planck differential operator (not encoding any specific trace condition at this stage).

The square bracket in the right-hand side of \eqref{eq:L_intermediate} only involves the pointwise values of $\xi$ inside $\Omega$ in the integral, together with its derivatives $\nabla\xi|_\O,\nabla^\tau\xi|_\pO$ through the norm $\|\xi\|_\rho$.
The last boundary term, on the other hand, only involves $\partial_N\xi|_\pO$.
We can therefore legitimately consider $\xi$ and $\xi_N=\partial_N\xi$ as independent variables in the supremum.
This yields
\begin{align*}
\lL(\rho,s)
&=\sup\limits_\xi\left[
\int_\O \xi(s-\mathcal Q^*\rho)
-\|\xi\|^2_\rho\right]
+\sup_{\xi_N} \int_\pO (\o-\g)\xi_N
\notag
\\
&=
\sup\limits_\xi\left[
\int_\O \xi(s-\mathcal Q^*\rho)
-\|\xi\|^2_\rho\right]
+
\iota_{\text{BC}}(\rho),
\end{align*}
where the convex indicator for the boundary condition is
$$
\iota_{\text{BC}}(\rho)
\coloneqq \begin{cases}
0 & \text{if }\o|_\pO=\g\\
+\infty &\text{else}
\end{cases},
\hspace{1cm}\rho=\o+\g.
$$
In order to make this Lagrangian fully explicit, recall that the norms $\|s\|_\rho,\|\xi\|_\rho$ were precisely constructed to be dual to each other.
It is then a simple exercise to realize that $\sup\limits_\xi\left[
\int_\O \xi s
-\|\xi\|^2_\rho\right]
=\frac 14\|s\|^2_\rho$,
whence
\begin{equation}
    \lL(\rho,s) = \frac 14\|s-\mathcal Q^*\rho\|^2_\rho + \iota_{\text{BC}}(\rho).
    \label{eq:L_explicit}
\end{equation}

Since the SRBD-Brownian Motion is reversible, the results from \cite{MR3269725} strongly suggest that the large-deviation cost~\eqref{eq:L_explicit} should induce a gradient flow structure.
We now show that this induced structure coincides with the Otto-type gradient flow from Subsection~\ref{ssec:Otto}.
Consider a curve $\rho=(\rho_t)_{t\in[0,T]}$ with finite action $I_T(\rho)=\int_0^T\lL(\rho_t,\dot\rho_t)\rd t<\infty$. The indicator $\iota_{\text{BC}}$ guarantees in particular that $\o_t\big|_{\pO}=\g_t$ for a.e.\ $t$, hence one simply reads off of \eqref{eq:elliptic_id} that the tangent element $-\mathcal Q^*\rho_t\in T_{\rho_t}\pP^*$ is represented by the scalar potential $\phi=-D_\rho\Ent\,(\rho_t)=-\log \rho_t$ (which is unambiguously defined precisely because of the trace condition $\o_t|_\pO=\g_t$).
Equivalently, this means that the gradient $\qQ^*\rho_t=-\grad\,\Ent(\rho_t)$ is well-defined and represented by the potential $D_\rho\Ent(\rho_t)=\log\rho_t=\1_\Oi\log\o_t+\1_\pO\log\g_t$.
As a consequence
$$
\lL(\rho_t,\dot\rho_t)=\frac 14\|\dot\rho_t-\qQ^*\rho_t\|^2_{\rho_t}
=
\frac 14\|\dot\rho_t+\grad\,\Ent(\rho_t)\|^2_{\rho_t}
\label{eq:ldp lagrangian}
$$
and the chain rule \eqref{eq:chain_rule_F} holds without boundary terms as
$$
\langle \dot\rho_t,\qQ^*\rho_t\rangle_{\rho_t}
=
-\langle \dot\rho_t,\grad\,\Ent(\rho_t)\rangle_{\rho_t} = -\frac d{dt}\Ent(\rho_t).
$$

Expanding squares, we end up with
\begin{multline*}
I_T(\rho)=\int_0^T\frac 14\|\dot\rho_t-\qQ^*\rho_t\|_{\rho_t}^2\rd t
\\
=\frac 12 \int_0^T\left( \frac 12\|\dot\rho_t\|_{\rho_t}^2+ \frac 12\left\|-\qQ^*\rho_t\right\|_{\rho_t}^2 -\langle \dot\rho_t,\qQ^*\rho_t\rangle_{\rho_t}\right)\rd t
\\
=\frac 12 \int_0^T\left( \frac 12\|\dot\rho_t\|_{\rho_t}^2+ \frac 12\left\|-D_\rho\Ent(\rho_t)\right\|_{\rho_t}^2 +\frac d{dt}\Ent(\rho_t)\right)\rd t.
\end{multline*}
Setting
$$
\Psi(\rho,s)\coloneqq \frac 12 \|s\|^2_\rho
\qqtext{and}
\Psi^*(\rho,\xi)\coloneqq \frac 12 \|\xi\|^2_\rho,
$$
which are by construction Legendre conjugates of one another, we finally obtain
\begin{align*}
I_T(\rho)<\infty
& \implies
\o|_\pO=\g
\\
& \implies
I_T(\rho)=
\frac 12\left[\Ent(\rho_T)-\Ent(\rho_0)+ \int_0^T(\Psi(\rho_t,\dot \rho_t)+\Psi^*(\rho_t,-D_\rho\Ent(\rho_t))\rd t\right].
\end{align*}
As in \cite{MR3269725}, this means exactly that the Fokker-Planck equation $\partial_t\rho=\qQ^*\rho$ is the $\Psi/\Psi^*$ gradient-flow of $\Ent_\mu$, at least formally on the ``submanifold'' of trace-matching densities $\o|_\pO=\g$.

\begin{remark}
\label{rmk:no_LDP?}
We stress again that the derivation of the large-deviation local cost~\eqref{eq:L_explicit} is formal, and actually has several significant problems.
First, the Hamiltonian $\hH$ from \eqref{eq:hdef} is the convex dual of the local cost $\lL$.
This produces a contradiction unless $\lL^{**}=\lL$, which necessarily fails if $\lL$ is not lower semicontinuous.
Second, on the path level, and by definition of large-deviation principles, the full rate function $I_0+I_T$ must be lower semi-continuous (for some suitable path topology in which one is trying to establish said LDP).
However, at least for the uniform topology, $I_T=\int_0^T\lL$ fails to be lower semi-continuous if $a<1$. 
And finally, a correct topology needs to be chosen for the LDP so that the trace operator $\omega\coloneqq\rho \mrest_{\Oi}$ is well defined.
Nevertheless, even if \eqref{eq:L_explicit} does not represent the true large-deviation function, it does define a valid variational formulation of the Wentzell heat flow~\eqref{eq:FPE}, and as such it can be used to derive an induced gradient system in the sense of \cite{MR3269725}.
\end{remark}

\subsection{Benamou-Brenier Formula and associated Wasserstein distance}
\label{ssec:BB}
As for the standard Neumann heat flow, the Otto calculus induces an intrinsic Benamou-Brenier distance between probability measures.
A remarkable feature of the resulting intrinsic distance for the Wentzell heat flow is that it depends on $a$ in a nontrivial fashion only  for  $a\geq 1$, showcasing again the phase transition across $a=1$.

Let us define the Lagrangian
\begin{equation}
\label{eq:def_La}
L_a(x,v)
\coloneqq
\1_{\Oi}(x)\frac{|v|^2}{2} + \1_{\dOm}(x)\left(\frac{|v|^2}{2a}+\iota_{T_x\dOm}(v)\right) \, 
\end{equation}
where $\iota_{T_x\dOm}: \R^d \to  \{0,\infty\}$ is the convex indicator of the tangent plane  $T_x\dOm\subset \R^d$ at a point $x\in \pO$.
One checks easily that the lower semi-continuous envelope is
\begin{equation}
\label{eq:def_La_bar}
\underline{L}_a(x,v)
=
\begin{cases}
L_a(x,v) & \text{if }a\geq 1 , \\
L_1(x,v) & \text{if }a<1 .
\end{cases}
\end{equation}
We have then
\begin{proposition}
Fix $a>0$.
Then
\begin{equation}
\label{def:a_distance}
    d^2_a(x,y)
    \coloneqq
    \inf_{\substack{X_0=x\\X_1=y} } \int_0^1 L_a(X_t,\dot{X}_t)\,\rd t
    =
    \min_{\substack{X_0=x\\X_1=y} } \int_0^1 \uL_a(X_t,\dot{X}_t)\,\rd t
\end{equation}
is a squared distance on $\O$ for which $(\O,d_a)$ is geodesic and complete.
Here we mean that the infimum is taken over paths $X\in H^1([0,1];\O)$ that run in $\O$ for all times.
Moreover for any $\rho_0, \rho_1 \in \mathcal P(\Om)$ there holds
\begin{equation}
\label{eq:WBB=Wda}
\inf\left\{ \left. \int_0
\hoch 1 \|V_t\|\hoch 2 _{\rho_t} \, \rd t \, \right | \partial_t \rho_t +\dive(\rho_tV_t)=0 \text{ with endpoints }\rho_0,\rho_1 \right\} = \wW\hoch 2 _a(\rho_0, \rho _1),
\end{equation}
where $\wW_a$ is the Wasserstein distance with respect to the cost function $c(x,y)=d^2_a(x,y)$.
\end{proposition}
Note that for $a<1$ it follows from \eqref{eq:def_La_bar} and \eqref{def:a_distance} that $d_a$ is simply the induced Euclidean distance on $\O$, and as such does not depend on $a$.
On the other hand for $a\geq 1$ the induced geometry depends on $a$ in a nontrivial fashion: in \cite[section 6]{casteras2025largedeviationsstickyreflectingbrownian}  the dynamical problem \eqref{def:a_distance} is studied in details, and it is shown in particular that geodesics follow a Snell-Descartes law and always make an angle $\alpha$ with the normal to the boundary such that $\sin^2(\alpha)=\frac 1a$.
\begin{proof}
Let us begin with \eqref{def:a_distance}.
We first point out that, although $L_a,\uL_a$ encode the constraint that $v\in T_x\pO$ at boundary points, the path functional appearing in \eqref{def:a_distance} actually only sees $v=\dot X_t$, the derivative of a path running in $\O$.
By classical properties of Sobolev functions one has $\dot X_t\in T_{X_t}\pO$ for a.e. $t$ such that $X_t\in\pO$, thus in practice one can ignore the convex indicator.
In the same spirit, it is not too difficult to check that $X\mapsto \int _0^1 \uL_a(X_t,\dot X_t)\rd t$ is coercive and lower-semicontinuous for the weak $H^1$ convergence, see \cite[section 6]{casteras2025largedeviationsstickyreflectingbrownian}.
If $a\geq 1$ then $\uL_a=L_a$ and the statement is vacuous, hence we only consider $a<1$.
Note that $\uL_a\leq L_a$ by definition of a lower-semicontinuous envelope, so that automatically $\min \int \uL_a= \inf \int \uL_a\leq \inf\int L_a$.
In order to get the reverse inequality, consider first a path $X=(X_t)$ never touching the boundary.
Then $L_a(X_t,\dot X_t)=\frac 12|\dot X_t|^2=\uL_a(X_t,\dot X_t)$ for all $t$ and clearly $\int L_a=\int \uL_a$ along this particular path.
If now $X$ has positive sojourn time along $\pO$, it is easy to construct a path $X^\eps=(X^\eps_t)$ that never touches the boundary (except possibly at $t=0,1$ if either $x$ or $y$ lies on the boundary) and such that $
\int  |\dot X^\eps_t|^2\rd t\leq \int|\dot X_t|^2\rd t+o_\eps(1)$.
(For example take first $ X^\eps_t$ to be the projection of $X_t$ on $\Omega_{2\eps}=\{x\in \Oi,\,\operatorname{dist}(x,\pO)\geq 2\eps\}$ and then adjust to match the endpoint constraints if needed, by connecting $x,X^\eps_0$ and $y, X^\eps_1$ in small time for an arbitrary small cost.
We omit the details for brevity).
As a consequence for any path $X$ and arbitrary $\eps$ we can produce a path $X^\eps$ with same endpoints, such that
\[
\int_0^1 L_a(X^\eps_t,\dot X^\eps_t)\rd t =
\int_0^1 \frac 12|\dot X^\eps_t|^2\rd t
\leq \int_0^1 \frac 12|\dot X_t|^2\rd t+o_\eps(1)
=\int_0^1 \uL_a(X_t,\dot X_t)\rd t+o_\eps(1).
\]
(Here the first equality holds due to $X^\eps_t\in \Oi$ for all $t\in(0,1)$.)
This entails the reverse inequality $\inf \int L_a\leq \inf \int \uL_a$ and \eqref{def:a_distance} follows.
It also follows from \eqref{def:a_distance} that $d_a$ is a geodesic distance, geodesics being simply minimizing curves in $d^2_a=\min\uL_a$.
Since $\min\{\frac 12,\frac 1{2a}\}|v|^2\leq L_a(x,v)\leq \max\{\frac 12,\frac 1{2a}\}|v|^2$ we see that $d_a$ is bi-Lipschitz equivalent to the induced Euclidean distance $d_\O$, hence $(\O,d_a)$ is also complete.
Note that this bi-Lipschitz equivalence shows that absolutely continuous paths with respect to either $d_a$ or the Euclidean distance $d_\O$ coincide, hence we only write $AC$ below without further mention.

Turning now to the Benamou-Brenier representation \eqref{eq:WBB=Wda}, let us first show that
\begin{equation}
\label{eq:Wa_leq_WBB}
    \wW_a^2(\rho_0,\rho_1)\leq \int_0^1\int L_a(x,V_t(x))\rho_t(\rd x)\rd t
\end{equation}
for any solution of the continuity equation.
If $(\rho,V)$ has infinite kinetic energy the desired inequality is trivial.
Otherwise one can apply the superposition principle \cite[Theorem 8.2.1]{MR2401600} and get a path-measure $\eta\in\pP(C([0,1];\O))$ such that:
\begin{enumerate}[(i)]
    \item 
    $\eta$ is supported on $AC^2$ paths solving $\dot X_t=V_t(X_t)$ in the integral sense
    \item
    $\rho_t=\mathsf e_t\pf\eta$ for all $t$, where $\mathsf e_t(X)=X_t$ is the time-$t$ evaluation
\end{enumerate}
As a consequence the plan $\pi=(\mathsf e_0,\mathsf e_1)\pf\eta\in \pP(\O\times\O)$ is an admissible coupling between $\rho_0,\rho_1$.
By definition of $d_a$ we have $d^2_a(X_0,X_1)\leq \int_0^1L_a(X_t,\dot X_t)\rd t$ for any path $X$, hence
\begin{multline*}
\iint d^2_a(x_0,x_1)\pi(\rd x_0,\rd x_1)
=
\int d^2_a(X_0,X_1)\eta(\rd X)
\\
\leq 
\int \int_0^1 L_a(X_t,\dot X_t)\rd t\,\eta(\rd X)
=
\int \int_0^1 L_a(X_t,V_t( X_t))\rd t\,\eta(\rd X)
\\
=\int_0^1\int L_a(X_t,V_t( X_t))\eta(\rd X)\,\rd t
=\int_0^1\int_\O L_a(x,V_t( x))\rho_t(\rd x)\,\rd t
\end{multline*}
and \eqref{eq:Wa_leq_WBB} follows.

We claim now that
\begin{equation}
\label{eq:WBB_leq_Wa}
    \inf\limits_{\rho,V}\int_0^1\int L_a(x,V_t(x))\rho_t(\rd x)\rd t
    \leq \wW_a^2(\rho_0,\rho_1),
\end{equation}
where the infimum runs over all solutions of the continuity equation.
As $(\O,d_a)$ is a separable, complete geodesic space, so is $(\pP(\O),\wW_a)$. Hence given any $\rho_0,\rho_1$ there exists a geodesic $\rho=(\rho_t)_{t\in[0,1]}$ with constant-speed $|\rho'|_{\wW_a}(t)=\wW_a(\rho_0,\rho_1)$ for a.e. $t\in (0,1)$.
In order to obtain a Lagrangian representation of this geodesic we apply again a superposition principle, this time \cite[Theorem 5]{lisini2007characterization}: there exists a path-measure $\eta\in \pP(C([0,1];\O))$ such that
\begin{enumerate}[(i)]
    \item 
    $\eta$ is concentrated on $AC^2$ curves
    \item
    $\rho_t=\mathsf e_t\pf\eta$ for all $t$
    \item
    the metric speeds superpose as
    \begin{equation}
    \label{eq:metric_speed_superpose_Lisini}
    |\rho'|^2_{\wW_a}(t) = \int |X'|^2_{d_a}(t) \eta(\rd X)
    \end{equation}
\end{enumerate}
Conditioning $X_t=x$ for fixed $(t,x)$, or, more precisely, by disintegration, there exists a Borel family of probability measures $\eta_{t,x}(\rd X)$ such that
\[
\int\int\phi(t,X)\eta(\rd X)\,\rd t
=
\int\int \left(\int\phi(t,X)\eta_{t,x}(\rd X)\right)\rho_t(\rd x)\rd t.
\]
We set accordingly
\begin{equation}
    V_t(x)
    \coloneqq \int
    \dot X_t \eta_{t,x}(\rd X),
\end{equation}
i.e. we superpose all the speeds of Lagrangian particles passing through $x$ at time $t$.
It is not difficult to check that this particular pair $(\rho,V)$ solves the continuity equation.
Since $\uL_a(x,v)$ is convex in $v$ we get by Jensen's inequality
\begin{align*}
    \int\int \uL_a(x,V_t(x))\rho_t(\rd x)\rd t
    &=
    \int\int \uL_a\left(x, \int \dot X_t \eta_{t,x}(\rd X)\right)\rho_t(\rd x)\rd t
    \\
    &\leq
    \int\int \int \uL_a\left(x, \dot X_t \right)\eta_{t,x}(\rd X)\rho_t(\rd x)\rd t
    \\
    &=
    \int\int \int \uL_a\left(X_t, \dot X_t \right)\eta_{t,x}(\rd X)\rho_t(\rd x)\rd t
    \\
    &=
    \int \int \uL_a\left(X_t, \dot X_t \right)\eta(\rd X)\rd t.
\end{align*}
Owing now to \eqref{def:a_distance} one checks that, at the Lagrangian level, the $d_a$-metric speed of any $AC^2$ path is actually
\begin{equation*}
    |X'|^2_{d_a}(t)=\lim\limits_{h\to 0} \frac{d^2_a(X_t,X_{t+h})}{h^2}=\uL_a(X_t,\dot X_t)
    \qqtext{for a.e. }t\in (0,1).
\end{equation*}
(Here it is key that the relaxed Lagrangian $\uL_a$ is involved and not $L_a$.)
From the previous inequality and \eqref{eq:metric_speed_superpose_Lisini} we see that
\begin{align*}
\int\int \uL_a(x,V_t(x))\rho_t(\rd x)\rd t
&\leq 
\int \int \uL_a\left(X_t, \dot X_t \right)\eta(\rd X)\rd t
\\
&=
\int \int |X'|^2_{d_a}(t)\eta(\rd X)\rd t
=
\int |\rho'|^2_{\wW_a}(t)\rd t
=
\wW_a^2(\rho_0,\rho_1),
\end{align*}
where the last equality simply comes from the fact that $\rho=(\rho_t)_{t\in[0,1]}$ was a $\wW_a$ geodesic between $\rho_0,\rho_1$ to begin with.
This is almost \eqref{eq:WBB_leq_Wa}, with however the slight issue that $\uL_a$ is featured instead of $L_a$.
If $a\geq 1$ we have $L_a=\uL_a$ and the proof is complete, so from now on we only consider $a<1$.

As a last step, for this particular curve $\rho$ with the above Benamou-Brenier cost, we will construct an $\eps$-modification $(\rho^\eps, V^\eps)$ with endpoints $\rho^\eps_0=\rho_0,\rho^\eps_1=\rho_1$ and such that
$$
\int\int L_a(x, V^\eps_t(x))\rho^\eps_t(\rd x)\rd t
\leq
\int\int \uL_a(x,V_t(x))\rho_t(\rd x)\rd t+o_\eps (1)
$$
for arbitrarily small $\eps$.
This will clearly entail \eqref{eq:WBB_leq_Wa}.

To achieve this, take a small $\eps$ and let $\O_\eps=\{x\in \O\vert\quad d(x,\pO)\geq \eps\}$.
For small $\eps$ this is a smooth domain, and one can construct a $C^1$ retract $T^\eps(x)$ of $\pO$, i.e.\ a map $\O\to\O$ such that
\begin{itemize}[(i)]
    \item 
    $T_\eps(\O)\subseteq \O_{\eps}$
    \item
    $T_\eps$ is the identity on $\O_{2\eps}$
    \item
    The Jacobian $|D_x T^\eps(x)|\leq (1+o(1))$ uniformly in $x\in \O$ as $\eps\to 0$.
\end{itemize}
(This can be constructed explicitly using normal coordinates in the interior $2\eps$-neighborhood of $\pO$ and we omit the details for brevity.)
Setting
$$
\rho^\eps_t\coloneqq T^\eps\pf\rho_t
\qqtext{and}
V^\eps_t\coloneqq (D_x T^\eps \cdot V_t)\circ (T^{\eps})^{-1}
$$
it is an easy exercise to check, by properties of the push-forward operation, that $(\rho^\eps,V^\eps)$ solve the continuity equation.
Observing that $\rho^\eps$ is supported in $\O_\eps$, i.e.\ away from $\pO$, we compute
\begin{multline*}
    \int\int L_a(x,V^\eps_t(x))\rho^\eps_t(\rd x)\rd t
    =
    \int\int \frac 12|V^\eps_t(x)|^2\rho^\eps_t(\rd x)\rd t
    \\
    =
    \int\int \frac 12|(D_x T^\eps\cdot V_t)\circ(T^{\eps})^{-1}(x)|^2\rho^\eps_t(\rd x)\rd t
    \\
    =\int\int \frac 12|(D_x T^\eps\cdot V_t)(x)|^2\rho_t(\rd x)\rd t
    \leq 
    (1+o_\eps(1))\int\int \frac 12|V_t(x)|^2\rho_t(\rd x)\rd t
    \\
    =\int\int  \uL_a(x,V_t(x))\rho_t(\rd x)\rd t+o_\eps(1)
\end{multline*}
due to $|D_x T^\eps(X)|\leq 1+o_\eps(1)$ uniformly.
The key point is here that we forced the curve $\rho^\eps$ to remain supported away from the boundary, thus $L_a(x,v)=\uL(x,v)$.

It only remains to restore the time marginals, since at this stage $\rho^\eps_0=T^\eps\pf\rho_0,\rho^\eps_1=T^\eps\pf\rho_1$ may differ from $\rho_0,\rho_1$.
For this we proceed almost exactly as in the beginning of the proof for \eqref{def:a_distance}, but at the Eulerian level: we first rescale $s=\eps+(1-2\eps)t$ and modify $(\rho^\eps_t,V^\eps_t)_{t\in [0,1]}\leadsto (\rho^\eps_s,V^\eps_s)_{s\in[\eps,1-\eps]}$ to run in slightly shorter time with virtually no extra cost (a simple $\frac 1{1-2\eps}$ multiplicative factor appears due to scaling properties of our quadratic Lagrangian).
In order to connect $\rho_0,\rho_1$ to $\rho^\eps_\eps=T^\eps\pf\rho_0$ and $\rho^\eps_{1-\eps}=T^\eps\pf\rho_1$ one can then explicitly construct a flow moving particles in the tubular $2\eps$-neighborhood $\O\setminus\O_{2\eps}$ in the normal direction only, over a distance at most $|T^\eps(x)-x|\leq 2\eps$ in time $s\in [0,\eps]$ and $s\in[1-\eps,1]$, thus with speed $\mathcal O_\eps(1)$.
This induces an additional cost not exceeding $\int_0^\eps\mathcal O_\eps(1)\,\rd s=o_\eps(1)$ and the proof is complete.
\end{proof}

\begin{remark} Alternatively, one could pursue a duality approach as in \cite{MR1760620} to prove the Benamou-Brenier formula. In our setting this would lead to the Hamilton-Jacobi equation
\begin{equation*}
\partial_t \phi+ H(x,\nabla \phi)=0
\hspace{1cm}
\text{with Hamiltonian }
H(x,p)=
\begin{cases}
 \frac{1}{2}|p|^2  &\text{ if } x\in\Oi , \\
 \frac{a}{2}| p^\tau|^2 & \text{ if } x\in\dOm ,
\end{cases}
\end{equation*}
where $p^\tau \in T_x\dOm\subset \R^d$ denotes the tangential part of $p$ at $x \in \dOm$. 
This Hamiltonian is exactly the Legendre conjugate of the Lagrangian \eqref{eq:def_La}, and accordingly the tangential projection $p^\tau$ stems from the convex indicator $\iota_{T_x\pO}(v)$.
As before, the failure of lower semi-continuity of the Lagrangian for $a<1$ causes an issue in the subsequent use of the expected Hopf-Lax formula.
In particular, the duality approach involves an effective double-Legendre transform of the Lagrangian in the Benamou-Brenier cost functional.  
Our proof via superposition principle avoids this implicit convexification.
\end{remark}


\section{No JKO-Wasserstein Gradient Flow Representation of Wentzell Heat Flow for weak Boundary Diffusion} \label{section:metricsection}

With regards to the purely metric theory of gradient flows, it was shown in the preceding work~\cite{MR4901547} for the  case of $a=1$  that the Wentzell heat flow \eqref{eq:FPE}  is a metric Wasserstein gradient flow, i.e.
\begin{equation*}
\dot\rho = -\grad_\wW \Ent_{\mu}(\rho),
\end{equation*}
in the sense of curves of maximal slopes \cite{MR2401600},  where $\wW$ is the classical quadratic Wasserstein distance over $\Om$. For $a>1$ the  analogous result is expected to hold, using the  Wasserstein metric $\wW_a$ induced by the point distance $d_a$ from \eqref{def:a_distance} instead.
By way of contrast, in the regime of weak boundary diffusion $a\in(0,1)$ we argue here that no such representation can hold.   
\subsection{Heuristics via Schilder-type large-deviations for SRBD-Brownian motion}
Before presenting our rigorous argument we recall another popular method~\cite{leonard2007large} for the  derivation of a metric Wasserstein gradient flow structure of the Fokker-Planck equation associated to a generic, reversible diffusion process $(X_t)_t$ on a smooth space  with invariant measure $\mu$, using a family of scaled Schr\"odinger problems.
As usual one picks  the relative entropy $\Ent_\mu(\cdot)$ with respect to $\mu$ as  driving functional. 
For a canonical choice of a Wasserstein-like distance one considers the short-time expansion for the  $\eps$-Schr\"{o}dinger problem with respect to the slowed-down diffusion process 
\begin{equation*}
\underset{\eps\to0}{\Gamma-\lim}\ \min_P \{ \eps \Ent_{R^\eps}(P)\ |\ P_0\sim\rho_0,P_1\sim\rho_1\}.
\end{equation*}
where $R^\eps \in \pP(C([0,1];\Om))$ is the path measure corresponding to the process on time scale $\eps$.
Assuming a Schilder-type large-deviation principle for the underlying process, the limit $\eps \to 0$  produces an optimal transportation functional on the pair $(\rho_0, \rho_1)$, in which the large-deviation rate function is the  cost \cite{MR2873864} and which is therefore a natural candidate ``metric'' for the  gradient flow representation.\\

Following this heuristic in our case for BM with SRBD in the regime $a\in(0,1]$ we can use  a Schilder-type large-deviation principle for BM with SRBD ($a\in(0,1]$) with the rate function 
\begin{equation*}
I(X):=
\begin{cases}
\frac{1}{4}\int_0^1 |\dot{X}_t|^2 \rd t & \text{ if } X\in H^1,
\\
+ \infty &\text{else},
\end{cases}
\end{equation*}
on the space of continuous paths, which was shown recently in~\cite{casteras2025largedeviationsstickyreflectingbrownian} for the special case when $\Om$ is a half space.
Note carefully that this rate function $I$ does \emph{not} depend on $a\in (0,1]$.
Furthermore, the limit of $\eps$-Schr\"{o}dinger problems for the slowed-down process can be identified as a Monge-Kantorovich problem with respect to the rate function $I$ from the Schilder-type large-deviation principle as cost, see~\cite{MR2873864}.
As a result,  one obtains up to a constant factor the standard quadratic Wasserstein distance regardless of the value of $a\in (0,1)$, leading to the paradoxical suggestion of the same metric Wasserstein gradient structure for all $a$ in this range. 

\subsection{A one parameter family of processes with common intrinsic metric}

The independence of the rate function above from the parameter $a\in (0,1]$ relates directly to the  aforementioned work by Karl-Theodor Sturm~\cite{MR1477263}. In fact all SRBD-Brownian motions for $a \in (0,1]$ share the same  intrinsic metric, which we state as a separate result.

\begin{proposition}
\label{lem:intmetric=euclid metric}
   Fix  $a\in(0,1]$ and let
   $$
   d_i(x,y)\coloneqq \sup\Big\{f(x)-f(y),\quad f\in \D_0\Big\}
   $$
   be the intrinsic distance on $\O$ induced by the Dirichlet form \eqref{eq:DF}, where
\begin{equation}
\label{eq:def_D0}
\dD_0 :=\Big\{ f\in \dD_b^c
\quad \text{s.t.}\quad
2\eE(fh,f) - \eE(f^2,h) \le \|h\|_{L^1(\mu)}\ \text{ for any } h\in \dD_b^c\Big\}
\end{equation}
with the ``bounded'' domain $\dD_b^c:=\dD(\eE)\cap C_b(\Om)$.
Then $d_i$ coincides with the distance $d_\Om$ on $\O$
$$
d_i(x,y)=d_{\O}(x,y)
:=\min\limits_{\substack{X_0=x, X_1=y\\ X_t \in \Om \, \forall t \in [0,1]}}\int_0^1|\dot X_t|\rd t,
$$
which for convex $\Om$ is just the Euclidean distance.
\end{proposition}
\begin{proof}
It is not difficult to check that the Dirichlet domain is
$$
\D(\eE)=\Big\{
f\in H^1(\O^\circ)\qtext{s.t.}f|_{\pO}\in H^1(\pO)
\Big\},
$$
where the restriction $f|_{\pO}$ is understood as $f|_{\pO}=\tr f$ in the sense of Sobolev boundary traces.
This is clearly independent of $a$, and since $\mu\in \mathcal P(\O)$ also does not depend on $a$ we see that $\dD_b^c=\D(\eE)\cap C_b(\Om)$ is independent of the value of $a$.
Moreover, computing explicitly
$$
2\eE(fh,f) - \eE(f^2,h)
=
\int_{\Oi} |\nabla f|^2 h c\,\rd\la+ a \int_{\dOm} |\nabla^\tau f|^2 hc\,\rd\si
$$
and recalling that $\mu=c(\lambda+\sigma)$ we have
$$
\dD_0 = \left\{ f\in \dD_b^c\quad \text{s.t.}\quad  |\nabla f|^2\le 1\ \lambda\text{-a.e. in } \Om \text{ and } |\nabla^\tau f|^2 \le \frac{1}{a}\ \si\text{-a.e. on } \dOm\right\}.
$$
For a given $f\in H^1(\O^\circ)$ with $|\nabla f|\leq 1$ $\lambda$-a.e. the boundary trace $\tr f=f|_{\pO}$ is automatically Lipschitz with $|\nabla^\tau f|\leq 1$, hence for $\frac 1a\geq 1$ the constraint $|\nabla^\tau f|^2\leq\frac 1a$ in the above expression of $\D_0$ is redundant and
\begin{equation}
 \D_0
 =
 \left\{ f\in \dD_b^c\quad \text{s.t.}\quad  |\nabla f|\le 1\ \lambda\text{-a.e. in } \Om\right\}
 \label{eq:D0_explicit}
\end{equation}
is finally independent of $a$.\\
We claim that
$$
\D_0=\Lip_1(|\cdot|)\coloneqq
\Big\{
f:\quad |f(x)-f(y)|\leq |x-y|
\text{ locally in } \Om\Big\}.
$$
Indeed, by \eqref{eq:D0_explicit} and the usual Rademacher property in $\O^\circ\subset\R^d$ we see that 
$\D_0\subseteq \Lip_1(|\cdot|)$.
Conversely, if $f\in\Lip_1(|\cdot|)$ then in particular $f\in H^1(\O^\circ)$ with $|\nabla f(x)|\leq 1$ a.e.
Moreover by standard results the $H^1$ boundary trace coincides with the restriction (in the $C_b$ sense), $f|_{\pO}=\tr f$.
Viewing $\pO$ as a Riemannian manifold with induced distance $d_\pO(x,y)\geq |x-y|$, we have, for $x,y\in \pO$, that $|f(x)-f(y)|\leq |x-y|\leq d_\pO(x,y)$.
In particular $f|_{\pO}$ is $1$-Lipschitz for the $d_\pO$ distance, thus also $\nabla^\tau f\in L^2(\pO,\sigma)$ and therefore $\Lip_1(|\cdot|)\subset \D_0$.

Recalling that 
\[
d_\Omega(x,y)=\sup\Big\{f(x)-f(y)\,\vert\,\,f\in \Lip_1(|\cdot|)\Big\}
,\]
it follows immediately from $\D_0=\Lip_1(|\cdot|)$ that
\begin{equation*}
d_i(x,y)
=\sup\Big\{f(x)-f(y):\quad f\in \Lip_1(|\cdot|)\Big\}
=d_\O(x,y). 
    \end{equation*}
\end{proof}

\subsection{Non-Existence of  Wasserstein gradient flow for weak boundary diffusion}
In our rigorous statement below, a metric gradient flow of $\mathrm{\Ent}_\mu$ with respect to the quadratic Wasserstein metric induced by a given point metric $d$ on $\Om$ is understood as  any limit of an infinitesimal JKO-scheme for vanishing step size~\cite{JKO98}. More precisely, given an initial datum $\rho_0=f_0 \mu\in \dD(\mathrm{Ent}_\mu)$ and a time step $h>0$ we consider the sequence $\rho_n^h=f_n^h \mu$ defined by the recursive variational problem
\begin{equation*}
\rho_n^h \in \argmin_{\rho\in \pP(\Om)} \left\{ \frac{1}{2h} \wW_2^2 (\rho,\rho_{n-1}^h) + \mathrm{Ent}_\mu(\rho)\right\},
\end{equation*} 
and set $\rho^h(t)=f^h(t)\mu := \rho_n^h$ if $t\in((n-1)h,nh]$.
Here $\wW$ is the quadratic Wasserstein metric based on $d$. A curve $(\rho_t)$ is then an \textit{infinitesimal JKO-gradient flow} if it holds that $\rho_t= \lim_{h \to 0}\rho^h(t)$ w.r.t.\ some metric metrizing weak convergence, locally uniformly in $t \geq 0$. 
(In \cite{MR2401600} this is precisely  referred to as a \emph{Generalized Minimizing Movement}, and is one among many possible notions of a metric gradient flow.)
\smallskip

The key to disproving the existence of a gradient flow representation for the Wentzell heat flow with weak boundary diffusion  is the recovery of the metric from the heat kernel, which is possible in a variety of situations.
The class of metrics considered here will be characterised by the following definition, motivated from the results in~\cite{MR1988472,MR4319821,GTT}, using the concept of the Cheeger energy associated to a metric measure space (cf.\ \cite{MR1708448,MR3152751}). 

\begin{definition}
\label{def:cheeger_regular}
We call a metric measure space $(\Om,d,\mu)$ a \emph{Cheeger-regular space} if
\begin{itemize}
\item the metric measure space  $(\Om, d,\mu)$ is infinitesimally Hilbertian
\item the Cheeger energy $(\mathrm{Ch},W^{1,2}(\Om,d,\mu))$  is a conservative regular Dirichlet form
\item for all open sets $A,B\subset \Om$ it holds that
\begin{equation}\label{eq:cheegerreg}
\inf_{x\in A} \inf_{y\in B} d(x,y)=: d(A,B) = d_{\mathrm{Ch}}(A,B),
\end{equation}
where
\begin{equation}\label{def:dch}
d_{\mathrm{Ch}}(A,B)\coloneqq \sup_{f\in \dD_0} \left\{ \essinf_{x\in B} f(x) - \esssup_{y\in A} f(y)\right\}
\end{equation}
with the Cheeger energy $\mathrm{Ch}$ on $\Om$ and
\begin{equation*}
\dD_0 :=\Big\{ f\in \dD_b\ \big\vert\quad\ 2 \mathrm{Ch}(fh,f) - \mathrm{Ch}(f^2,h) \le \|h\|_1 \text{ for any } h\in \dD_b\Big\}
\end{equation*}
and $\dD_b:=\dD(\mathrm{Ch})\cap L^\infty(\mu)=W^{1,2}(\Om,d,\mu)\cap  L^\infty(\mu)$.
\end{itemize}
\end{definition}

\begin{remark}
Set distances of the form~\eqref{def:dch} appear in integrated versions of Varadhan’s short-time asymptotic formula of Hino-Ramirez cf.~\cite{MR1988472}, stating that the heat semigroup associated to a strongly local Dirichlet form $\eE$ on some locally compact topological space $(\mathcal X, \tau)$ with invariant measure $\mu$ satisfies  \begin{equation}\label{eq:hr}
\lim_{t\to 0}t \log P_t(A,B) = - \frac{d_{\eE}(A,B)^2}{2}
\end{equation}
for all measurable sets $A,B$, where
\begin{equation*}
P_t(A,B)\coloneqq \int_A P_t \1_B \rd\mu.
\end{equation*}
and the ``integrated'' distance is
\begin{equation}\label{eq:disetdist}
d_{\eE}(A,B)\coloneqq \sup_{f\in \dD_0} \left\{ \essinf_{x\in B} f(x) - \esssup_{y\in A} f(y)\right\}
\end{equation}
with
\begin{equation*}
\dD_0 :=\Big\{ f\in \dD_b\ \big\vert\quad\ 2\eE(fh,f) - \eE(f^2,h) \le \|h\|_1 \text{ for any } h\in \dD_b\Big\}
\end{equation*}
and $\dD_b:=\dD(\eE)\cap L^\infty(\mu)$.
It needs to be pointed out that  the set distance $d_{\eE}$ is in general not induced by a point metric on $\Om$ (see e.g.\ also~\cite{MR4319821} for a more detailed discussion).
This motivates the condition above on Cheeger-regular spaces that this set distance $d_{\mathrm{Ch}}$ be induced by the given point metric $d$.
In~\cite{MR4319821,MR4498486} it has been shown that condition~\eqref{eq:cheegerreg} is fulfilled in $\mathrm{RCD}(K,\infty)$ spaces, $\mathrm{RQCD}(Q,K,N)$ spaces as well as quasi-regular strongly local Dirichlet spaces with admissible extended pseudo-distance $d$ under the so-called Rademacher and Sobolev-to-Lipschitz conditions.
\end{remark}

\begin{theorem}\label{thm:general}
Let $\eE$ be a strongly local (quasi-)regular Dirichlet form  on some locally compact Polish topological space $(\mathcal X, \tau)$ with invariant measure $\mu$.  If any Fokker-Planck flow induced by $\eE$ can be obtained as infinitesimal JKO-gradient flow of $\Ent_\mu$ with respect to the Wasserstein metric induced by some point metric $d$ on $\Om$ such that 
\begin{itemize}
\item $(\Om,d,\mu)$ is a Cheeger-regular space
\item the topology induced by $d$ coincides with $\tau$
\item $\mu$ has full support,
\end{itemize}
then the set distance $d$ has to coincide with the set distance $d_{\eE}$ defined in equation~\eqref{eq:disetdist} for open sets $A,B$ with $\mu(A)>0,\mu(B)>0$. Furthermore, there can be at most one point metric $d$ fulfilling this.
\end{theorem}

\begin{proof}
By \cite[Corollary 8.2]{MR3152751} the heat flow induced by the Cheeger-energy on $(\Om,d,\mu)$ coincides with the metric gradient flow (in the sense of an infinitesimal JKO-gradient flow) of $\Ent_\mu$ with respect to the Wasserstein metric with cost $d^2$. By assumption the latter
also coincides with the (Fokker-Planck) heat flow induced by the initial Dirichlet form $\eE$.
As the metric measure space $(\Om,d,\mu)$ is assumed to be Cheeger-regular we have
\begin{equation}
\lp d_{\eE}(A,B) \rp^2 = \lim_{t\to 0} (-2t \log P_t(A,B)) = \lp d_{\mathrm{Ch}}(A,B)\rp^2 = \lp d(A,B) \rp^2.
\end{equation}
Here the first and second equality are due to application of an integrated version of Varadhan's short-time asymptotic formula of Hino and Ram\'{i}rez, cf.~\cite[Theorem 1.1]{MR1988472} and the last equality is by definition of a Cheeger-regular metric measure space.\\
Furthermore, assume there are two distinct point metrics $d,d'$ fulfilling the assumptions of the theorem. As the assumptions assure that balls with respect to the metric $d$ or $d'$ have positive measure, we may argue as follows: For $x,y\in\Om$ with $d(x,y)\neq d'(x,y)$ we may construct sufficiently small open balls $B(x)$ around $x$ and $B(y)$ around $y$ such that
$
d(B(x),B(y))= d_{\eE}(B(x),B(y)) = d'(B(x),B(y))$ but $d(B(x),B(y))\neq d'(B(x),B(y))$, which is a contradiction.
\end{proof}

\begin{lemma}\label{lem:di}
The set distance $d_{\eE}$ associated with the Dirichlet form~\eqref{eq:DF} corresponding to BM with SRBD for $a\in(0,1]$ is independent of the value of $a$.
\end{lemma} 
\begin{proof}
Recall that for all measurable sets $A,B$
\begin{equation*}
d_{\eE}(A,B):= \sup_{f\in \dD_0} \left\{ \essinf_{x\in B} f(x) - \esssup_{y\in A} f(y)\right\}
\end{equation*}
where analogous to Proposition~\ref{lem:intmetric=euclid metric}
\begin{equation*}
 \D_0
 =
 \left\{ f\in \dD_b\quad \text{s.t.}\quad  |\nabla f|\le 1\ \lambda\text{-a.e. in } \Om\right\}
\end{equation*}
is independent of $a$.
\end{proof}


\begin{theorem}
For $a\in(0,1)$ there is no metric $d$ on $\Om$ such that $(\Om,d,\mu)$ is a Cheeger-regular space and such that  the Fokker-Planck equation~\eqref{eq:FPE} can be obtained as the infinitesimal JKO gradient flow of the  $\mu$-entropy w.r.t.\ the quadratic Wasserstein distance induced by $d$.\end{theorem}
\begin{proof} 
Assume there is a metric $d$ fulfilling the conditions of the proposition.
Since $\mu=c(\lambda+\sigma)$ has full support in $\O$ Theorem~\ref{thm:general} guarantees that the induced set distance $d$ coincides with the set distance $d_{\eE}$ on open sets of positive measure and $d$ is unique. However, by Lemma~\ref{lem:di} the set distance $d_{\eE}$ is independent of the value of $a\in(0,1]$.
The unique point metric $d$ whose corresponding set distance $d$ coincides with $d_{\eE}$ is the one corresponding to the case $a=1$.
The metric gradient flow with respect to the $d^2$-Wasserstein distance cannot describe the Fokker-Planck equation~\eqref{eq:FPE} for $a\in(0,1)$, as the Fokker-Planck equation clearly depends on the value of $a$ while the metric and driving functional do not.
\end{proof}

 \nocite{MR2665414,GTT}
\subsection*{Acknowledgement}
L.M. was funded by FCT - Funda\c{c}\~ao para a Ci\^encia e a Tecnologia through a personal grant 2020/00162/CEECIND (DOI 10.54499/2020.00162.CEECIND/CP1595/CT0008) and project UIDB/00208/2020 (DOI 10.54499/UIDB/00208/2020)
\bibliographystyle{abbrv}
\bibliography{bib}

\begin{thebibliography}{10}

\bibitem{adams2013large}
S.~Adams, N.~D., M.~Peletier, and J.~Zimmer.
\newblock Large deviations and gradient flows.
\newblock {\em Philosophical Transactions of the Royal Society A: Mathematical,
  Physical and Engineering Sciences}, 371(2005):20120341, 2013.

\bibitem{MR2401600}
L.~Ambrosio, N.~Gigli, and G.~Savar\'e.
\newblock {\em Gradient flows in metric spaces and in the space of probability
  measures}.
\newblock Lectures in Mathematics ETH Z\"urich. Birkh\"auser Verlag, Basel,
  second edition, 2008.

\bibitem{MR3152751}
L.~Ambrosio, N.~Gigli, and G.~Savar\'e.
\newblock Calculus and heat flow in metric measure spaces and applications to
  spaces with {R}icci bounds from below.
\newblock {\em Invent. Math.}, 195(2):289--391, 2014.

\bibitem{MR4096131}
A.~Aurell and B.~Djehiche.
\newblock Behavior near walls in the mean-field approach to crowd dynamics.
\newblock {\em SIAM J. Appl. Math.}, 80(3):1153--1174, 2020.

\bibitem{bormann2025cheegertypeinequalitydriftlaplacian}
M.~Bormann.
\newblock A {C}heeger-type inequality for the drift {L}aplacian with
  {W}entzell-type boundary condition, 2025.
\newblock arXiv: 2503.16093.

\bibitem{brw}
M.~Bormann, M.~von Renesse, and F.-Y. Wang.
\newblock Functional inequalities for {B}rownian motion on manifolds with
  sticky-reflecting boundary diffusion.
\newblock {\em Probab. Theory and Related Fields}, 2024.

\bibitem{casteras2025largedeviationsstickyreflectingbrownian}
J.-B. Casteras, L.~Monsaingeon, and L.~Nenna.
\newblock Large deviations for sticky-reflecting {B}rownian motion with
  boundary diffusion, 2025.
\newblock arXiv: 2501.11394.

\bibitem{MR4901547}
J.-B. Casteras, L.~Monsaingeon, and F.~Santambrogio.
\newblock Sticky-reflecting diffusion as a {W}asserstein gradient flow.
\newblock {\em J. Math. Pures Appl. (9)}, 199:Paper No. 103721, 32, 2025.

\bibitem{MR1708448}
J.~Cheeger.
\newblock Differentiability of {L}ipschitz functions on metric measure spaces.
\newblock {\em Geom. Funct. Anal.}, 9(3):428--517, 1999.

\bibitem{dawsont1987large}
D.~A. Dawson and J.~G{\"a}rtner.
\newblock Large deviations from the {M}c{K}ean-{V}lasov limit for weakly
  interacting diffusions.
\newblock {\em Stochastics: An International Journal of Probability and
  Stochastic Processes}, 20(4):247--308, 1987.

\bibitem{MR4319821}
L.~Dello~Schiavo and K.~Suzuki.
\newblock Rademacher-type theorems and {S}obolev-to-{L}ipschitz properties for
  strongly local {D}irichlet spaces.
\newblock {\em J. Funct. Anal.}, 281(11):Paper No. 109234, 63, 2021.

\bibitem{MR4498486}
L.~Dello~Schiavo and K.~Suzuki.
\newblock Sobolev-to-{L}ipschitz property on {${QCD}$}-spaces and applications.
\newblock {\em Math. Ann.}, 384(3-4):1815--1832, 2022.

\bibitem{Dudley2002}
R.~M. Dudley.
\newblock {\em Real Analysis and Probability}.
\newblock Cambridge University Press, Cambridge, UK, 2 edition, 2002.

\bibitem{MR2260560}
J.~Feng and T.~G. Kurtz.
\newblock {\em Large deviations for stochastic processes}, volume 131 of {\em
  Mathematical Surveys and Monographs}.
\newblock American Mathematical Society, Providence, RI, 2006.

\bibitem{MR2665414}
A.~Figalli and N.~Gigli.
\newblock A new transportation distance between non-negative measures, with
  applications to gradients flows with {D}irichlet boundary conditions.
\newblock {\em J. Math. Pures Appl. (9)}, 94(2):107--130, 2010.

\bibitem{GTT}
N.~Gigli, L.~Tamanini, and D.~Trevisan.
\newblock Viscosity {S}olutions of {H}amilton–{J}acobi {E}quation in
  $\mathrm{RCD}({K},\infty)$ {S}paces and {A}pplications to {L}arge
  {D}eviations.
\newblock {\em Potential Analysis}, 2024.

\bibitem{MR3613700}
M.~Grothaus and R.~Vo{\ss}hall.
\newblock Stochastic differential equations with sticky reflection and boundary
  diffusion.
\newblock {\em Electron. J. Probab.}, 22:Paper No. 7, 37, 2017.

\bibitem{MR1988472}
M.~Hino and J.~A. Ram\'irez.
\newblock Small-time {G}aussian behavior of symmetric diffusion semigroups.
\newblock {\em Ann. Probab.}, 31(3):1254--1295, 2003.

\bibitem{MR126883}
N.~Ikeda.
\newblock On the construction of two-dimensional diffusion processes satisfying
  {W}entzell's boundary conditions and its application to boundary value
  problems.
\newblock {\em Mem. Coll. Sci. Univ. Kyoto Ser. A. Math.}, 33:367--427,
  1960/61.

\bibitem{JKO98}
R.~Jordan, D.~Kinderlehrer, and F.~Otto.
\newblock The variational formulation of the {F}okker-{P}lanck equation.
\newblock {\em SIAM J. Math. Anal.}, 29(1):1--17, 1998.

\bibitem{MR4575023}
V.~Konarovskyi.
\newblock Coalescing-fragmentating {W}asserstein dynamics: particle approach.
\newblock {\em Ann. Inst. Henri Poincar\'{e} Probab. Stat.}, 59(2):983--1028,
  2023.

\bibitem{konarovskyi2021spectralgapestimatesbrownian}
V.~Konarovskyi, V.~Marx, and M.~von Renesse.
\newblock Spectral gap estimates for {B}rownian motion on domains with
  sticky-reflecting boundary diffusion, 2021.
\newblock arXiv: 2106.00080.

\bibitem{MR4704529}
V.~Konarovskyi and M.-K. von Renesse.
\newblock Reversible coalescing-fragmentating {W}asserstein dynamics on the
  real line.
\newblock {\em J. Funct. Anal.}, 286(8):Paper No. 110342, 60, 2024.

\bibitem{leonard2007large}
C.~L{\'e}onard.
\newblock A large deviation approach to optimal transport.
\newblock arXiv preprint arXiv:0710.1461, 2007.

\bibitem{MR2873864}
C.~L\'eonard.
\newblock From the {S}chr\"odinger problem to the {M}onge-{K}antorovich
  problem.
\newblock {\em J. Funct. Anal.}, 262(4):1879--1920, 2012.

\bibitem{lisini2007characterization}
S.~Lisini.
\newblock Characterization of absolutely continuous curves in {W}asserstein
  spaces.
\newblock {\em Calculus of variations and partial differential equations},
  28(1):85--120, 2007.

\bibitem{MR3269725}
A.~Mielke, M.~A. Peletier, and D.~R.~M. Renger.
\newblock On the relation between gradient flows and the large-deviation
  principle, with applications to {M}arkov chains and diffusion.
\newblock {\em Potential Anal.}, 41(4):1293--1327, 2014.

\bibitem{MR1842429}
F.~Otto.
\newblock The geometry of dissipative evolution equations: the porous medium
  equation.
\newblock {\em Comm. Partial Differential Equations}, 26(1-2):101--174, 2001.

\bibitem{MR1760620}
F.~Otto and C.~Villani.
\newblock Generalization of an inequality by {T}alagrand and links with the
  logarithmic {S}obolev inequality.
\newblock {\em J. Funct. Anal.}, 173(2):361--400, 2000.

\bibitem{quattrocchi2024variationalstructuresfokkerplanckequation}
F.~Quattrocchi.
\newblock Variational structures for the {F}okker--{P}lanck equation with
  general {D}irichlet boundary conditions.
\newblock {\em arXiv preprint arXiv:2403.07803}, 2024.

\bibitem{MR1477263}
K.-T. Sturm.
\newblock Is a diffusion process determined by its intrinsic metric?
\newblock {\em Chaos Solitons Fractals}, 8(11):1855--1860, 1997.

\bibitem{MR929208}
S.~Takanobu and S.~Watanabe.
\newblock On the existence and uniqueness of diffusion processes with
  {W}entzell's boundary conditions.
\newblock {\em J. Math. Kyoto Univ.}, 28(1):71--80, 1988.

\bibitem{MR121855}
A.~D. Ventcel'.
\newblock On boundary conditions for multi-dimensional diffusion processes.
\newblock {\em Theor. Probability Appl.}, 4:164--177, 1959.

\bibitem{MR287612}
S.~Watanabe.
\newblock On stochastic differential equations for multi-dimensional diffusion
  processes with boundary conditions. {II}.
\newblock {\em J. Math. Kyoto Univ.}, 11:545--551, 1971.

\end{thebibliography}
\end{document}